\documentclass[a4paper, twoside, 12pt]{article}
\usepackage[latin9]{inputenc}
\usepackage[leqno,tbtags]{amsmath}
\usepackage{amssymb,amsthm,comment,xypic}
\usepackage{mathrsfs}
\usepackage[english]{babel}
\usepackage{makeidx,layout}

\makeindex

\newcommand{\ie}{{i.\,e.,\ }}

\newcommand{\eg}{{e.\,g.,\ }}

\newcommand{\PAR}[1]{{ \left(  {#1} \right)  }}
\newcommand{\BRC}[1]{{ \left\{ {#1} \right\} }}

\newcommand{\ANG}[1]{{ \left<  {#1} \right>  }}

\newcommand{\RA}{\rightarrow}

\newcommand{\MT}{\mapsto}

\renewcommand{\~}{\cong}

\newcommand{\I}{{\mathrm{i}}}

\newcommand{\Ga}{\alpha}

\newcommand{\Gz}{\zeta}
\newcommand{\Gt}{\vartheta}

\newcommand{\Gl}{\lambda}
\newcommand{\Gr}{\varrho}
\newcommand{\Gs}{\sigma}
\newcommand{\Gf}{\varphi}
\newcommand{\Go}{\omega}
\newcommand{\GG}{\Gamma} 
\newcommand{\GL}{\Lambda}
\newcommand{\GS}{\Sigma}

\newcommand{\Fg}{{\mathfrak{g}}}

\newcommand{\Ft}{{\mathfrak{t}}}
\newcommand{\Fu}{{\mathfrak{u}}}

\newcommand{\R}{\mathbb{R}}
\newcommand{\C}{\mathbb{C}}
\newcommand{\Z}{\mathbb{Z}}
\newcommand{\N}{\mathbb{N}}

\newcommand{\PS}{\mathbb{P}}

\newcommand{\orb}{{\mathcal{O}}}

\newcommand{\Cinf}{C^\infty}

\newcommand{\UE}{{\mathcal{U}}}

\newcommand{\Tr}{\mathrm{T}}

\newcommand{\IO}[1]{\frac{\I}{#1}}

\newcommand{\OH}{\frac{1}{2}}

\newcommand{\COV}{\nabla}

\newcommand{\dd}{\mathrm{d}}

\newcommand{\df}{\dd f}

\renewcommand{\dh}{\dd h}

\newcommand{\dPh}{\dd \Phi}

\newcommand{\TF}[2]{{\frac{\dd {#1}}{\dd {#2}}}}

\newcommand{\TFpt}[2]{ { {\left. {\TF{}{#1}} \right|}_{#2} } }

\newcommand{\SSE}{\subseteq}

\renewcommand{\Im}{\mathrm{Im}}
\DeclareMathOperator{\ID}{id}
\DeclareMathOperator{\Ad}{Ad}

\DeclareMathOperator{\END}{End}
\DeclareMathOperator{\HOM}{Hom}

\DeclareMathOperator{\Mat}{Mat}

\DeclareMathOperator{\tr}{tr}

\DeclareMathOperator{\Hol}{Hol}

\DeclareMathOperator*{\complsum}{\widehat{\bigoplus}}

\def\haken{\mathbin{\hbox to 6pt{%
                 \vrule height0.4pt width5pt depth0pt
                 \kern-.4pt
                 \vrule height6pt width0.4pt depth0pt\hss}}}

\theoremstyle{definition}
\newtheorem{Def}{Definition}[section]
\newtheorem*{Def*}{Definition}

\newtheorem*{DefF*}{D\'efinition}

\newtheorem*{Exm*}{Example}

\newtheorem*{ExmF*}{Exemple}

\newtheorem*{ExmD*}{Beispiel}

\newtheorem{Rem}[Def]{Remark}
\newtheorem*{Rem*}{Remark}

\newtheorem*{RemF*}{Remarque}

\newtheorem*{RemD*}{Bemerkung}

\newtheorem*{Not*}{Notation}

\theoremstyle{plain}
\newtheorem{Thm}[Def]{Theorem}
\newtheorem*{Thm*}{Theorem}

\newtheorem*{ThmF*}{Th\'eor\`eme}

\newtheorem{Lemma}[Def]{Lemma}
\newtheorem*{Lemma*}{Lemma}

\newtheorem*{LemmaF*}{Lemme}

\newtheorem{Prop}[Def]{Proposition}
\newtheorem*{Prop*}{Proposition}

\newtheorem*{Coro*}{Corollary}

\newtheorem*{CoroF*}{Corollaire}

\newtheorem*{CoroD*}{Korollar}

\newtheorem*{Ques*}{Question}

\numberwithin{equation}{section}

\title{On the Geometry and Quantization\\
	of Symplectic Howe Pairs}
\author{
{\bf Carsten Balleier\footnote{supported by the German Research Foundation (DFG) and the Universit\'e 
franco-allemande (DFH-UFA) via the IRTG 1133 ``Geometry and Analysis of Symmetries"} 
}\\
{\bf Tilmann Wurzbacher
}\\
     {\small Laboratoire de Math\'ematiques et Applications de Metz}\\
             {\small Universit\'e Paul Verlaine-Metz et C.\,N.\,R.\,S.}\\
                 {\small Ile du Saulcy, F-57045 Metz, France}\\
                 {\small \tt{wurzbacher@math.univ-metz.fr}  }  
  }

\begin{document}

\date{October 5, 2009}

\maketitle


\abstract
We study the orbit structure and the geometric quantization of a pair
of mutually commuting hamiltonian actions on a symplectic manifold.
If the pair of actions fulfils a {\em symplectic Howe condition},
we show that there is a canonical correspondence between the
orbit spaces of the respective moment images. Furthermore, we show that
reduced spaces with respect to the action of one group are symplectomorphic
to coadjoint orbits of the other group. In the K\"ahler case we show
that the linear representation of a pair of compact Lie groups
on the geometric quantization of the manifold is then equipped with a 
representation-theoretic Howe duality.

\section{Introduction}
In representation theory, Howe duality (\cite{Howe-remarks}) is well-known.
In the model situation of a product $G_1 \times G_2$ of two compact
connected Lie groups linearly represented on a finite dimensional complex
vector space $U$ ($\Gr : G_1 \times G_2 \RA GL(U)$), we say that the
representation satisfies the {\em Howe condition} or is equipped with a {\em
Howe duality} if there is a subset $\mathcal{D}$ of $\widehat{G_1}$, the set
of equivalence classes of irreducible complex representations of $G_1$, and an
injective map $\GL : \mathcal{D} \RA \widehat{G_2}$ such that
\[
	(*) \qquad\qquad\qquad
	U \~ \bigoplus_{\Ga \in \mathcal{D}} V_\Ga \otimes W_{\GL(\Ga)},
	\qquad\phantom{0}
\]
where $V_\Ga$ represents a class $\Ga$ in $\mathcal{D}$ and $W_{\GL(\Ga)}$
represents the class $\GL(\Ga)$. Denoting for $k = 1,2$ the restriction
of the representation $\Gr$ to $G_k$ by $\Gr_k$ (as well as the
corresponding Lie algebra representation), condition $(*)$ is equivalent
to
\[
	(**) \qquad\qquad
	Z_{\END(U)}(\Gr_i(\UE \Fg_i)) = \Gr_j(\UE \Fg_j)
	\qquad \textnormal{for } i\ne j,
\]
where $\UE \Fg$ is the universal enveloping algebra (over $\C$) of a Lie
algebra $\Fg$, and for a set $S \SSE \END(U)$, $Z_{\END(U)}(S)$ is the centralizer
of $S$ in the endomorphisms of $U$. This equivalence is easily shown using
standard facts from representation theory, as found, \eg in \cite{GoodWall}.
Obviously, condition $(*)$ can be formulated for representations on a
complex vector space $U$ having countable dimension and -upon passing to an
appropriate completion on the right hand side- for representations on
Fr\'echet or more general topological vector spaces.

In the spirit of quantization/dequantization of physical and mathematical
structures (going back to the correspondence principle of Niels Bohr, 
see \cite{Bohr-coll3}), attempts to find analogous
notions in symplectic geometry have been made. This concerns analogues of
commutants ($G_\mathscr{V}$ in \cite{KKS-Calo}, polar actions in
\cite{OR-moment}) and dual pair notions (as in \cite{Wein-locstr}),
as well as the study of the orbit structure in particular cases like 
in, \eg \cite{Adams-orb-dual}. In this article, we study the setting of
commuting hamiltonian (proper) actions of two Lie groups on a symplectic
manifold. We define in this situation a {\em symplectic Howe condition},
naturally corresponding to the above condition $(**)$, namely
\[
	Z_{\Cinf(M)}(\Phi_i^*\Cinf(\Fg_i^*)) = \Phi_j^*\Cinf(\Fg_j^*)
		\qquad \textnormal{for } i \ne j.
\]
Here and in the sequel, the Lie algebra of a Lie group $G$ will always be
denoted by $\Fg$ and, for $A \SSE \Cinf(M)$,
$Z_{\Cinf(M)}(A) = \{f \in \Cinf(M) \:|\: \{f, a\} = 0 \:\forall a\in A\}$.\\

We analyse thoroughly its consequences on the orbit structure, obtaining
results that refine those of \cite{KKS-Calo}. We then prove an orbit
correspondence and show the conservation of integrality under the
correspondence. Furthermore we show that the Marsden-Weinstein reduced
spaces of the action of one group are equivariantly
symplectomorphic to coadjoint orbits in the moment image of the other group
action. To this setting we finally apply geometric quantization and obtain in 
the K\"ahler case a decomposition of the quantization of the initial
symplectic manifold as in $(*)$. That is, given a pair of commuting actions
by holomorphic transformation groups on a K\"ahler manifold satisfying the
symplectic Howe condition, the linear action of the pair on the holomorphic
quantization (the holomorphic sections of the quantizing holomorphic line
bundle) satisfies the representation-theoretic Howe condition. In the last
section we illustrate the results with some simple but instructive examples.

\section{Properties of Commuting Hamiltonian Actions}
The setting throughout this article is a smooth connected symplectic
manifold $(M, \Go)$ on which two smooth actions of the Lie groups $G_1$ and
$G_2$ are given.

Assume these actions commute and are hamiltonian such that each action
admits an equivariant moment map $\Phi_i : M \RA \Fg_i^*$ ($i = 1,2$), \ie
they satisfy $\dPh_i^\xi = \Go(\tau_i(\xi), \cdot)$ and
$\Phi_i(g\cdot z) = \Ad^*(g) \Phi_i(z) \:\forall z\in M, g \in G_i$.
Here, for any $\xi \in \Fg_i$, we denote by $\tau_i(\xi)$ the fundamental
vector field on $M$ associated to $\xi$ and by $\Phi_i^\xi$ the component
of the $i$-moment map in direction $\xi$. We first notice the following
easy lemma.

\begin{Lemma}
\label{LEM:invtPhi}
Let $(M, \Go)$ be a symplectic manifold. Let the Lie groups $G_1$ and $G_2$
act symplectically on $M\!$, admitting equivariant moment maps\/ $\Phi_1$ and
$\Phi_2$. Assume these actions commute. Then the product group $G_1 \times
G_2$ acts symplectically on $M$ with an equivariant moment map, and notably
for any $\xi \in \Fg_1$ and any
$\eta \in \Fg_2$ the Poisson bracket of the moment components vanishes, \ie
\[
	\BRC{\Phi_1^\xi, \Phi_2^\eta} = 0.
\]
\end{Lemma}
\begin{proof} Obvious.
\vskip-5mm\end{proof}

In other words, the moment map of the first action is constant along the
connected components of the orbits of the second action, and vice versa. For
$G_1$ and $G_2$ (or, at least, their orbits) connected, this can be
rephrased as
\[
	\Phi_i^* \Cinf(\Fg_i^*) \SSE \Cinf(M)^{G_j}
		\qquad \textnormal{for } i \ne j.
\]

Here we observe the inclusion of the collective functions for the first
action in the set of invariants for the second action, and vice versa.
Let us now observe the following useful -and most certainly folkloristic- result.
\begin{Lemma}
\label{LEM:centr-invt}
Let $G$ be a connected Lie group and $(M, \Go)$ a symplectic manifold. If
$G$ acts symplectically with an equivariant moment map $\Phi$ on $M$, then
\[
	Z_{\Cinf(M)}(\Phi^*\Cinf(\Fg^*)) = \Cinf(M)^G.
\]
\end{Lemma}
\begin{proof}
Note that for $f \in \Cinf(M)$ and $\xi \in \Fg$, one has that $\xi^M(f) =
\{f, \Phi^\xi\}$ (where $\xi^M = \tau(\xi)$ denotes the fundamental vector
field on $M$ associated to $\xi \in \Fg$), and, furthermore, if $z \in M$, one
has $\{f, \Phi^\xi\}(z) = -\ANG{\xi, T_z\Phi(X_{f|z})}$ where
$\ANG{\cdot, \cdot}$ is the natural pairing between $\Fg$ and $\Fg^*$ and
$\Go(X_f, \cdot) = \df$, \ie $X_f$ is
the hamiltonian vector field associated to $f$.

Given now $f \in \Cinf(M)^G$ one has that for all $z \in M$,
$T_z\Phi(X_{f|z}) = 0$. If $h \in \Cinf(\Fg^*)$ and $z \in M$, we have
\[
	\{h \circ \Phi, f\}(z) = \dd(h \circ \Phi)_z(X_{f|z})
		= \dh_{\Phi(z)}(T_z\Phi(X_{f|z})) = 0.
\]

On the other hand, if $f \in Z_{\Cinf(M)}(\Phi^*\Cinf(\Fg^*))$ and $\xi \in
\Fg$, then $\xi^M(f) = \{f, \Phi^\xi\} = 0$, and thus by the connectedness
of $G$, $f \in \Cinf(M)^G$.
\end{proof}

\begin{Rem}
More subtle relations between centralizers of pullbacks via $\Phi$ and
invariant functions are shown in \cite{KL-centr} in the case that $G$ is
compact (often extending, in fact, to proper actions).
\end{Rem}

With the two preceding lemmata we immediately get that
\[
	\Phi_i^* \Cinf(\Fg_i^*) \SSE Z_{\Cinf(M)}(\Phi_j^* \Cinf(\Fg_j^*))
		\qquad \textnormal{for } i \ne j.
\]
Equality of the respective sets is obviously the classical analogue of
$(**)$ upon interpreting $\Phi^*\Cinf(\Fg^*)$ as the ``classical collective
observables (w.\,r.\,t.\ to a hamiltonian $G$-action on $M$)'' and $\Gr(\UE
\Fg)$ as the ``quantum collective observables (w.\,r.\,t.\ to a linear
$G$-representation)''.
Therefore, it is natural to study the following situation.

\begin{Def}
We say that two commuting hamiltonian actions satisfy the
{\em symplectic Howe condition} or form a {\em symplectic Howe pair} if
\[
	Z_{\Cinf(M)}(\Phi_i^*\Cinf(\Fg_i^*)) = \Phi_j^*\Cinf(\Fg_j^*)
		\qquad \textnormal{for } i \ne j
\]
is satisfied.
\end{Def}

\begin{Rem}
The above condition is clearly close to the notions of
dual pairs presented in Ch.\,11 of \cite{OR-moment}. However, we do not require
the moment maps to be submersions nor surjections when dealing with the
symplectic Howe condition.
\end{Rem}

Assuming further the properness of both actions, the symplectic Howe condition
has the following important consequence for the orbit structure of these
commuting actions.

\begin{Prop}
\label{PR:mmlevels}
Let commuting hamiltonian proper actions of the connected Lie groups $G_1$
and $G_2$  with equivariant moment maps $\Phi_1$ and $\Phi_2$
be given on the symplectic manifold $(M, \Go)$.
Then the symplectic Howe condition
\[
	Z_{\Cinf(M)}(\Phi_i^*\Cinf(\Fg_i^*)) = \Phi_j^*\Cinf(\Fg_j^*)
		\qquad \textnormal{for } i \ne j
\]
implies that
\[
	\forall z \in M \textnormal{ holds }
		\Phi_i^{-1}(\Phi_i(z)) = G_j \cdot z
		\qquad \textnormal{for } i \ne j,
\]
\ie the levels sets of the moment maps of one action are the orbits of the
other one.
\end{Prop}

\begin{proof}
Let $i, j \in \{1, 2\}$ such that $i + j = 3$. Fix $z \in M$ and denote
$\Phi_i^{-1}(\Phi_i(z)) = N_z$.
By Lemma~\ref{LEM:invtPhi} and the connectedness of $G_j$, the $G_j$-action
on $M$ preserves the subset $N_z \SSE M$.

Assume that $N_z \ne G_j \cdot z$.
Then there exists $z' \in N_z$ such that $G_j \cdot z' \cap G_j \cdot z =
\emptyset$.
Since the $G_j$-action is proper, there exists $f \in \Cinf(M)^{G_j}$ such
that $f(z') \ne f(z)$ (separation of orbits by invariant functions follows
easily from the slice theorem for proper actions).

By Lemma~\ref{LEM:centr-invt} and the symplectic Howe condition, we have
\[
	\Cinf(M)^{G_j} = Z_{\Cinf(M)}(\Phi_j^*\Cinf(\Fg_j^*))
			= \Phi_i^*\Cinf(\Fg_i^*),
\]
\ie there is a smooth function $h : \Fg_i^* \RA \R$ such that $f =
h \circ \Phi_i$. Thus we arrive at $f(z') = h(\Phi_i(z')) =
h(\Phi_i(z)) = f(z)$, contradicting our assumption. Consequently,
$\Phi_i^{-1}(\Phi_i(z)) = G_j \cdot z$.
\end{proof}

Knowing the level sets of the moment maps as precisely as above permits to
relate the stabilizers of the actions on $M$ and on the moment images
$\Phi_i(M)$.
\begin{Lemma}
\label{LEM:stab-rels}
Let $G_1$ and $G_2$ be Lie groups and $(M, \Go)$ be a 
symplectic manifold. Let hamiltonian actions of both groups on $M$ be given
which commute, and denote the equivariant moment maps by $\Phi_i :
M \RA \Fg_i^*$. Let $G_{12,z} = \{(g, g') \in G_1 \times G_2 \:|\:
(g, g') \cdot z = z\}$ be the stabilizer of a point $z \in M$ under the
simultaneous action of both groups, write $H_{1,z}$ and $H_{2,z}$ for the
projections of $G_{12,z}$ to the groups $G_1$ and $G_2$.

Then

\renewcommand{\theenumi}{\roman{enumi}}
\renewcommand{\labelenumi}{(\theenumi)}
\begin{enumerate}
\item $H_{1,z}$ is contained in $G_{1,\Phi_1(z)}$, the stabilizer of the
image of $z$ under the coadjoint action (\ie $G_{1,z} \SSE H_{1,z} \SSE
G_{1,\Phi_1(z)}$), and
\item if, furthermore, the level sets of $\Phi_1$ are actually
$G_2$-orbits, then $H_{1,z} = G_{1,\Phi_1(z)}$.
\end{enumerate}
The same statements hold for the indices $1$ and $2$ interchanged.
\end{Lemma}
\begin{proof}
\renewcommand{\theenumi}{\roman{enumi}}
\renewcommand{\labelenumi}{(\theenumi)}
\begin{enumerate}
\item Take $g \in H_{1,z}$, \ie there exists $g' \in G_2$ such that
$(g, g') \cdot z = z$. From this one obtains
\[
        \Phi_1(z) = \Phi_1((g, g') \cdot z) =
                g\cdot \Phi_1((e, g') \cdot z) = g \cdot \Phi_1(z),
\]
where $G_1$-equivariance and $G_2$-invariance of $\Phi_1$ have been used.

\item Now take $g \in G_{1,\Phi_1(z)}$, then $\Phi_1(g\cdot z) = \Phi_1(z)$,
\ie both $z$ and $g\cdot z$ lie in a level set of $\Phi_1$, which is by
assumption a $G_2$-orbit. Thus there exists $g' \in G_2$ such that $(g, g')
\cdot z = z$, which was to be shown.
\end{enumerate}

Of course, interchanging both actions does not alter the proof.
\end{proof}

The particular orbit structure that we have described yields now a
correspondence between the orbits in the images of the moment maps $\Phi_1$
and $\Phi_2$.

\begin{Thm}
\label{TH:orbcorr}
Let commuting hamiltonian proper actions of the connected Lie groups $G_1$
and $G_2$ with equivariant moment maps $\Phi_1$ and $\Phi_2$
be given on the symplectic manifold $(M, \Go)$.

If the symplectic Howe condition is satisfied, then:

\renewcommand{\theenumi}{\roman{enumi}}
\renewcommand{\labelenumi}{(\theenumi)}
\begin{enumerate}  
\item There is a bijection $\GL :  {\Phi_1}(M)/G_1 \RA \Phi_2(M)/G_2$, given by
	\[
		\GL(\orb_{\Ga_1}) = \Phi_2(\Phi_1^{-1}(\orb_{\Ga_1})),
	\]
	where $\Ga_1 \in \Phi_1(M)$, and $\orb_{\Ga_1} = \Ad^*(G_1){\Ga_1}$ is
	seen as	an element of $\Phi_1(M)/G_1$. To each pair $(\orb_{\Ga_1},
	\GL(\orb_{\Ga_1}))$ belongs a unique orbit $(G_1 \times G_2) \cdot z$
	in $M$ given by $\Phi_1^{-1}(\orb_{\Ga_1})=\Phi_2^{-1}(\orb_{\Ga_2})$.
	We say that orbits $(G_1 \times G_2) \cdot z$, $\orb_{\Ga_1}$ and
	$\orb_{\Ga_2}$ are in correspondence if $\orb_{\Ga_2} =
	\GL(\orb_{\Ga_1})$ and $(G_1 \times G_2) \cdot z =
	\Phi_1^{-1}(\orb_{\Ga_1}) = \Phi_2^{-1}(\orb_{\Ga_2})$.
\item If $\Phi_1$ and $\Phi_2$ are open maps, then $\GL$ is a homeomorphism.
	The same conclusion holds if the groups $G_1$ and $G_2$ both are
	compact and $\Phi_1$ and $\Phi_2$ both are closed maps. Finally, the
	same conclusion holds if for $i + j = 3$, $G_i$ is compact and
	$\Phi_i$ is closed, and	$\Phi_j$ is open.
\item Write $M_{\Ga_1} = \Phi_1^{-1}(\Ga_1)/G_{1,\Ga_1}$ and $M_{\Ga_2} =
	\Phi_2^{-1}(\Ga_2)/G_{2,\Ga_2}$ for the respective point reduced
	spaces at $\Ga_1 \in \Fg^*_1$ and $\Ga_2 \in \Fg^*_2$.
	These spaces can be described as coadjoint orbits of the other
	action, \ie there are, resp., $G_2$- and $G_1$-equivariant
	symplectomorphisms
	\[
		M_{\Ga_1} \RA \GL(\orb_{\Ga_1})\quad
	\textnormal{ and }\quad
		M_{\Ga_2} \RA \GL^{-1}(\orb_{\Ga_2}).
	\]
\item The symplectic reduced space $M_{(\Ga_1, \Ga_2)}$ for the joint
	action of $G_1 \times G_2$ is either a point (if $\orb_{\Ga_1}$ and
	$\orb_{\Ga_2}$ are in correspondence) or empty otherwise.
\end{enumerate}
\end{Thm}
\begin{proof}
In order to simplify indices, some statements will only be proved for one
action if the other case is analogous.
\renewcommand{\theenumi}{\roman{enumi}}
\renewcommand{\labelenumi}{(\theenumi)}
\begin{enumerate}  
\item By Prop.\,\ref{PR:mmlevels}, we know that for any $z \in M$, the level
	sets of both moment maps are orbits:
	$\Phi_i^{-1}(\Phi_i(z)) = G_j \cdot z$ ($i+j = 3$). This implies
	that the preimage of any coadjoint orbit in either moment image
	$\Phi_i(M)$ ($i = 1,2$) is exactly one orbit of the joint action
	of $G_1 \times G_2$ on $M$, \ie
\[
	\Phi_i^{-1}(\Ad^*(G_i)\Phi_i(z)) = (G_1 \times G_2) \cdot z,
\]
	from which it follows that $\GL$ is well-defined and bijective.

\item Note that all maps in the diagram
\[ \xymatrix{
	\Phi_1(M)/G_1 & \Phi_1(M)\ar[l]_{\pi_1}
		& M\ar[l]_{\Phi_1}\ar[r]^{\Phi_2}
			& \Phi_2(M)\ar[r]^{\pi_2} & \Phi_2(M)/G_2
} \]
	are continuous, and $\pi_1$ and $\pi_2$ are always open as well.

Let us denote $\GL : \Phi_1(M)/G_1 \RA \Phi_2(M)/G_2$ by $\GL_{12}$ and its
set-theoretic inverse by $\GL_{21}$. If $\Phi_1$ is open, then for $U$ open
in $\Phi_2(M)/G_2$ we have that $\GL_{21}(U) = (\pi_1 \circ \Phi_1)
(\Phi_2^{-1}(\pi_2^{-1}(U)))$ is an open set, \ie $\GL_{12}$ is continuous.
Similarly, if $G_1$ is compact and $\Phi_1$ is closed, we have for $A$
closed in $\Phi_2(M)/G_2$ that $\GL_{21}(A)$ is closed as well, \ie again
$\GL_{12}$ is continuous.

	The conclusions of \labelenumi\ follow easily.

\item Let $z \in M$ and $\Ga_2 = \Phi_2(z)$ its value under the moment map
	of the second action. Consider the restricted map $\Phi_{2|G_2
	\cdot z} : G_2 \cdot z \RA \orb_{\Ga_2}$, and recall that $G_2\cdot z
	= \Phi_1^{-1}(\Ga_1)$ for $\Ga_1 = \Phi_1(z)$. Recall from
	Lemma~\ref{LEM:stab-rels} that $G_{1,\Ga_1} = \{h \in G_1 \:|\:
	\exists g \in G_2 : (h, g)\cdot z = z\}$. This group acts on
	$G_2\cdot z$ and one has $M_{\Ga_1} = G_2 \cdot z / G_{1,\Ga_1}$.
	Thus $\Phi_{2|G_2 \cdot z}$ induces
\[
	\tilde{\Phi}_2 : M_{\Ga_1} \RA \orb_{\Ga_2},
\]
	which inherits from $\Phi_2$ smoothness and $G_2$-equivariance. It
	is clearly surjective and we now show that it is injective: Take
	$\Ga \in \orb_{\Ga_2}$, $\tilde{z}_1, \tilde{z}_2 \in G_2 \cdot z /
	G_{1,\Ga_1}$ so that $\Ga = \tilde{\Phi}_2(\tilde{z}_1) =
	\tilde{\Phi}_2(\tilde{z}_2)$. Now fixing preimages $z_1, z_2 \in
	G_2 \cdot z$ of $\tilde{z}_1, \tilde{z}_2$, they have the property
	$\Phi_2(z_1) = \Phi_2(z_2)$, and for some $h \in G_1$, $z_2 =
	h \cdot	z_1$ holds because the level sets of $\Phi_2$ are $G_1$-orbits.
	However, $z_2 = g \cdot z_1$ for some $g \in G_2$, which implies by
	part (ii) of Lemma~\ref{LEM:stab-rels} that $h
	\in G_{1,\Ga_1}$. Therefore, $\tilde{z}_1 = \tilde{z}_2$, and
	$\tilde{\Phi}_2$ is a bijection. Applying Sard's Theorem, one
	notices that a smooth equivariant bijection between
	finite-dimensional homogeneous spaces has a smooth inverse. Thus
	$\tilde{\Phi}_2$ is a $G_2$-equivariant diffeomorphism and
	it remains to show that $\tilde{\Phi}_2$ is a symplectomorphism.
	Denote by $i_{G_2\cdot z} : G_2 \cdot z \RA M$ the inclusion of $G_2
	\cdot z = \Phi_1^{-1}(\Ga_1)$ into the ambient manifold. We observe
	that the equivariance properties of $\Phi_2$ imply
	$i_{G_2\cdot z}^*\Go = (\Phi_{2|G_2 \cdot z})^* \Go^{\orb_{\Ga_2}}$,
	where $\Go^{\orb_{\Ga_2}}$ is the KKS symplectic form (see, \eg
	\cite{OR-moment},  Thm.\,4.5.31). But the
	symplectic form $\Go^{M_{\Ga_1}}$ on $M_{\Ga_1}$ is defined such that
	it also pulls back to $i_{G_2\cdot z}^*\Go = p^*\Go^{M_{\Ga_1}}$,
	via the quotient map $p : G_2 \cdot z \RA G_2 \cdot z / G_{1,\Ga_1}$.
	Consequently we have $p^*\Go^{M_{\Ga_1}} = (\Phi_{2|G_2 \cdot z})^*
	\Go^{\orb_{\Ga_2}} = p^*(\tilde{\Phi}_2^* \Go^{\orb_{\Ga_2}})$.
	Since $p$ is a surjective submersion,
	the coincidence of the pullbacks (to $G_2\cdot z$)
	implies that $\Go^{M_{\Ga_1}} = \tilde{\Phi}_2^* \Go^{\orb_{\Ga_2}}$,
	which was to be shown.

\item Let $\Phi = \Phi_1 \oplus \Phi_2$. Take $\Ga_1 \in \Fg_1^*$ and
	$\Ga_2 \in \Fg_2^*$. Then
\[
	\Phi^{-1}(\orb_{\Ga_1} \times \orb_{\Ga_2})
\]
	is empty if $\orb_{\Ga_2} \ne \GL(\orb_{\Ga_1})$. Otherwise, for any
	$z \in M$ such that $\Phi_1(z) \in \orb_{\Ga_1}$ and $\Phi_2(z) \in
	\orb_{\Ga_2}$, holds
\[
	\Phi^{-1}(\orb_{\Ga_1} \times \orb_{\Ga_2})/(G_1 \times G_2)
	\~ (G_1 \times G_2)\cdot z / (G_1 \times G_2).
\]
	This quotient is, of course, a point.
	
\end{enumerate}
\end{proof}

\begin{Rem}
\label{RK:freeact}
\renewcommand{\labelenumi}{(\theenumi)}
\begin{enumerate}  
\item Statement (i) is a special case of the singular symplectic leaf
	correspondence of Thm.\,11.4.4 in \cite{OR-moment}, but here obtained
	from the symplectic Howe condition on the pair of groups actions, in
	a spirit close to the non-singular
	correspondence of Thm.\,11.1.9 of the same reference.
\item 
	Take any $z \in M$ and $\Ga_1 = \Phi_1(z)$. Then the global
	ineffectivity of the $G_{1,\Ga_1}$-action on the $\Phi_1$-level
	containing $z$ is $I_{1,\Ga_1} = \{g_1 \in G_{1,\Ga_1} \:|\:
	g_1 \cdot z' = z' \:\:\forall z' \in \Phi_1^{-1}(\Ga_1)\}$, \ie the
	intersection of the stabilizers of all points in the level set. Here,
	$\Phi_1^{-1}(\Ga_1) = G_2 \cdot z$, hence $G_{1,z'} = G_{1,z} \SSE
	I_{1,\Ga_1}$ for all $z' \in \Phi_1^{-1}(\Ga_1)$, so $I_{1,\Ga_1} =
	G_{1,z}$. Therefore, the proper action of $G_{1,\Ga_1}$ on
	$\Phi_1^{-1}(\Ga_1)$ factorizes over a free and proper action of
	$G_{1,\Ga_1}/I_{1,\Ga_1}$, and thus the quotient is a smooth manifold.
\item Part (iv) of the preceding theorem implies that the $(G_1\times
	G_2)$-action on $M$ is multiplicity-free, \ie $\Cinf(M)^{G_1\times
	G_2}$ is a commutative Poisson algebra (compare \cite{GS-mfsp} and
	\cite{HW-coisoact}). This can also be directly deduced from the
	definition of a symplectic Howe pair and Lemma~\ref{LEM:centr-invt}.
\end{enumerate}
\end{Rem}

\section{Prequantization of Symplectic Howe Pairs}

We are now going to show that the orbit correspondence which was constructed
above behaves well under geometric prequantization. The first step is to show
that the integrality of coadjoint orbits is preserved under the
correspondence. To achieve this, we make use of having proved that the reduced
spaces coincide in our setting with coadjoint orbits. Let us recall the
following fact (Thm.\,3.2 of \cite{GS-geomqu}).

\begin{Thm} 
\label{TH:red-preq}
Let $G$ be a compact connected Lie group and $(M, \Go)$ be a
prequantizable symplectic manifold with hamiltonian $G$-action, equivariant
moment map $\Phi$ and line bundle $L(M)$ equipped with a connection $\COV$
on $L(M)$ whose curvature coincides with $\Go$. If the $G$-action on the
level set $\Phi^{-1}(0)$ is free, then the reduced space 
$M_0 = \Phi^{-1}(0)/G$ is a manifold and there exists a unique line bundle
$L(M_0)$ over $M_0$ such that
\[
	\pi^*L(M_0) = i^*L(M) \textnormal{ and } \pi^*\COV_0 = i^*\COV,
\]
where $\pi : \Phi^{-1}(0) \RA \Phi^{-1}(0)/G$ is the quotient map and $i :
\Phi^{-1}(0) \RA M$ the inclusion.
\end{Thm}

\begin{Not*}
If $(M, \Go)$ is a {\em prequantizable} symplectic manifold, \ie $\Go$ is an
integral form, we denote by
$L(M, \Go)$ a corresponding prequantum line bundle, \ie a complex line
bundle with first de\,Rham Chern class equal to $[\Go]$. This line bundle is
unique if there exist no torsion line bundles on $M$. We may omit $\Go$ if
there is no ambiguity about the symplectic form. For a coadjoint orbit
$\orb_\Ga$, we will write $L_\Ga = L(\orb_\Ga)$, the KKS symplectic
form $\Go^{\orb_\Ga}$ being understood.
\end{Not*}

By the shifting trick, any reduced space may be regarded as a reduced space
at $0$. More precisely, one has the following well-known result (see, \eg Thm.\,6.5.2
in \cite{OR-moment}).
\begin{Thm}
Let $G$ be a compact connected Lie group and $(M, \Go)$ be a prequantizable
symplectic manifold with hamiltonian $G$-action and equivariant moment map
$\Phi$. If $G_\Ga$ acts freely on $\Phi^{-1}(\Ga)$, then
the reduced space at $\Ga \in \Fg^*$, $\Phi^{-1}(\Ga)/G_\Ga$, is 
symplectomorphic to the reduction of $M \times \orb_{\Ga}^-$ at $0$, where
$\orb_{\Ga}^-$ denotes the coadjoint orbit through $\Ga$ with the KKS
symplectic form multiplied by $-1$.
\end{Thm}

Returning to the setting of Theorem~\ref{TH:orbcorr}, we conclude that any
reduced space $M_{\Ga_1}$ admits a line bundle $L(M_{\Ga_1})$ if the coadjoint
orbit $\orb_{\Ga_1}$ does. Recall from Remark~\ref{RK:freeact}(2) that the
quotient of $\Phi_1^{-1}(\Ga_1)$ by $G_{1,\Ga_1}$ is also given as the
quotient by the free ($G_{1,\Ga_1}/I_{1,\Ga_1}$)-action on
$\Phi_1^{-1}(\Ga_1)$ and hence $M_{\Ga_1}$ is smooth. Therefore, we can adapt
Theorem~\ref{TH:red-preq} to this quotient. Using the
symplectomorphism $\tilde{\Phi}_2 : M_{\Ga_1} \RA \orb_{\Ga_2}$, we
obtain the bundle $\PAR{\tilde{\Phi}_2^{-1}}^* L(M_{\Ga_1})$ over
$\orb_{\Ga_2}$.
As coadjoint orbits of compact connected Lie groups are simply connected, 
there are no non-trivial flat vector bundles, hence no torsion line bundles
over them. Therefore, for any $\Ga_1$
such that $(\orb_{\Ga_1}, \Go^{\orb_{\Ga_1}})$ is an integral symplectic
manifold, the bundle $\PAR{\tilde{\Phi}_2^{-1}}^* L(M_{\Ga_1})$ is the
unique prequantum line bundle over $\orb_{\Ga_2}$. Thus we have proved:

\begin{Prop}
\label{PR:intcorr}
Let $\orb_{\Ga_1}$ and $\orb_{\Ga_2}$ be two coadjoint orbits in
correspondence as in Thm.\,\ref{TH:orbcorr}(i) and assume furthermore
that $G_1$ and $G_2$ are compact. Then $(\orb_{\Ga_1},
\Go^{\orb_{\Ga_1}})$ is an integral symplectic manifold if and only if
$(\orb_{\Ga_2}, \Go^{\orb_{\Ga_2}})$ is an integral symplectic manifold, too.
\end{Prop}

\begin{Rem}
Let us recall that for $G$ a compact connected semisimple Lie group and $\Ga
\in \Fg^*$, $\orb_\Ga$ is an integral symplectic manifold if and only if
there exists a character $\chi_\Ga : G_\Ga \RA U(1)$ such that
$(\chi_\Ga)_{*e} = 2\pi\I\Ga$. Thus one defines:
\end{Rem}

\begin{Def}
\label{DEF:integrelem}
\renewcommand{\labelenumi}{(\theenumi)}
\begin{enumerate}
\item Let $G$ be a compact connected Lie group and $\Ga \in \Fg^*$. We call
$\Ga$ {\em integral} if there exists $\chi_\Ga : G_\Ga \RA U(1)$ such that
$(\chi_\Ga)_{*e} = 2\pi\I\Ga$.
\item If $\Ga$ is integral, we call $\orb_\Ga$ an {\em integral orbit}.
\end{enumerate}
\end{Def}

\begin{Rem}
Of course, there is a natural bijection between the set of integral orbits
and the set of dominant weights of $G$ (w.\,r.\,t.\ a fixed Weyl chamber
$\Ft^+$).
\end{Rem}

We observe that integral coadjoint orbits are integral symplectic manifolds
but the converse statement is not true in general because of the possible
presence of a positive-dimensional centre. Accordingly, we need to refine
Proposition~\ref{PR:intcorr}.

Let us from now on assume that the $(G_1\times G_2)$-action on $M$ comes with
a fixed ``linearization on $L$'', \ie there is given a fibrewise linear
$(G_1 \times G_2)$-action on $L = L(M, \Go)$  covering the
$(G_1 \times G_2)$-action on $M$. This is no restriction if one accepts to
replace $G_1$ and $G_2$ by finite coverings (see \cite{Duis-Spinc}, Prop.\,15.4).

\begin{Thm}
Let $\orb_{\Ga_1}$ and $\orb_{\Ga_2}$ be two coadjoint orbits in
correspondence as in Thm.\,\ref{TH:orbcorr}(i) and assume furthermore
that $G_1$ and $G_2$ are compact. Then $\orb_{\Ga_1}$ is
integral if and only if $\orb_{\Ga_2}$ is integral.
\end{Thm}

\begin{proof}
Let $\Ga \in \Fg_1^*$ be integral. Since the $G_2$-actions on $M$ and $L$
commute with the $G_1$-actions, the $G_2$-action on $L$ descends canonically
to a $G_2$-action on the line bundle $L(M_{\Ga_1})$ covering the natural
$G_2$-action on $M_{\Ga_1}$.

Using the $G_2$-equivariant symplectomorphism $\tilde{\Phi}_2 :
(M_{\Ga_1}, \Go^{M_{\Ga_1}}) \RA (\orb_{\Ga_2}, \Go^{\orb_{\Ga_2}})$ from
above, we have the bundle $(\tilde{\Phi}_2^{-1})^*L(M_{\Ga_1}) =
L(\orb_{\Ga_2})$ together with a $G_2$-invariant connection $\COV$ given by
pulling back the $G_2$-invariant connection induced on $L(M_{\Ga_1})$ by a
$(G_1 \times G_2)$-invariant connection on $L \RA M$. (See
\cite{GS-geomqu} for the construction of the connection on $L(M_{\Ga_1})$.)
We thus arrive at the following commuting diagram (where $p$ denotes the
bundle projection and $L(\orb_{\Ga_2})_{\Ga_2} = p^{-1}(\Ga_2)$):

\[\xymatrix{
	G_2 \times L(\orb_{\Ga_2}) \ar[r] \ar[d]_{\ID \times p}
					& L(\orb_{\Ga_2}) \ar[d]^{p}\\
	G_2 \times \orb_{\Ga_2} \ar[r] & \orb_{\Ga_2},
}\]
yielding a homomorphism $\chi : G_{2, \Ga_2} \RA U(L(\orb_{\Ga_2})_{\Ga_2}) =
U(1)$.

It remains to show that $\chi = \chi_{\Ga_2}$, \ie $\chi_{*e} =
2\pi\I\Ga_2$. Since the $G_2$-action on $L(\orb_{\Ga_2})$ and the connection
$\COV$ come from the reduced bundle with connection $L(M_{\Ga_1})$ and this
comes in turn from a connection on $L \RA M$, we have the usual ``Kostant
formula'' for the fundamental vector fields of the $G_2$-action on
$L(\orb_{\Ga_2})$. (Compare \cite{Kost-preq}, Thm.\,4.5.1.) More precisely,
given $\xi \in \Fg_2$, $\xi^{L(\orb_{\Ga_2})} = \widetilde{\xi^{\orb_{\Ga_2}}}
+ p^*(\ANG{\Phi_2^{\orb_{\Ga_2}}, \xi}) (2\pi\I)^{L(\orb_{\Ga_2})}$, where
$\xi^N$ denotes the fundamental vector field associated to $\xi$ on a
$G_2$-manifold $N$, $\widetilde{X}$ denotes the $\COV$-horizontal lift of a
vector field $X$ on $\orb_{\Ga_2}$ to $L(\orb_{\Ga_2})$, $2\pi\I \in \I\R \~
\Fu(1)$ also has a fundamental vector field on the bundle by the canonical
$U(1)$-action on it, and $\ANG{\Phi_2^{\orb_{\Ga_2}}, \xi}$ is the
$\xi$-component of the $G_2$-moment map on $\orb_{\Ga_2}$.

For $\xi \in \Fg_{2,\Ga_2}$, the field $\xi^{L(\orb_{\Ga_2})}$ is now tangent
to the $p$-fibre $L(\orb_{\Ga_2})_{\Ga_2}$ over $\Ga_2$ and equals there
$(2\pi\I)^{L(\orb_{\Ga_2})} \ANG{\Ga_2, \xi}$.
Thus $\chi_{*e} : G_{2,\Ga_2} \RA U(1)$ is equal to $\TFpt{t}{0} e^{2\pi\I
\ANG{\Ga_2, \xi} t} = 2\pi\I \ANG{\Ga_2, \xi}$.
\end{proof}

\section{Geometric Quantization of Symplectic Howe Pairs}

In this section we show that the symplectic Howe condition for the actions
on $M$ implies the representation-theoretic Howe condition for the linear
representation on the geometric quantization of $(M, \Go)$ in the K\"ahler
case. Thus we assume throughout this section that $(M, \Go)$ is K\"ahler,
\ie there is given a complex structure $J$ such that the associated
riemannian metric $g(\cdot, \cdot) = \Go(\cdot, J\cdot)$ is positive
definite and hermitean. Furthermore we assume that the acting groups are
compact and connected and that the moment maps are admissible in the sense
of \cite{Sj-redmult} (p.\,109).
\begin{Def}
Let $G$ be a compact connected Lie group acting holomorphically on a K\"ahler
manifold $M$ with equivariant moment map $\Phi$.
Choose a $G$-invariant inner product on $\Fg$ with corresponding norm and
define the function $\mu = \|\Phi\|^2$. Let $F_t$ be the gradient flow of
$-\mu$. A moment map is called {\em admissible} if for every $z \in M$, the
path of steepest descent $F_t(z)$ through $z$ ($t \ge 0$) is contained in a
compact set.
\end{Def}
\begin{Rem}
Examples of admissible moment maps are all proper moment maps and the moment
map of the natural linear $U(n)$-action on $\C^n$.
\end{Rem}

One then has the following result (Thm.\,2.20 of  \cite{Sj-redmult}).
\begin{Thm}
\label{TH:QRcomm}
Let $G$ be a compact connected Lie group acting holomorphically on a K\"ahler
manifold $M$ with admissible equivariant moment map $\Phi$. Suppose this
action extends to an action of the complexification $G^\C$ of $G$. Then for
every dominant weight $\Ga$ of $G$, the space of holomorphic sections
$\GG_\mathrm{\!hol}(M_\Ga, L_\Ga)$ of the prequantum line bundle $L_\Ga$ 
over the symplectic reduced space $M_\Ga$ is naturally isomorphic to
$\HOM_G(V_\Ga, \GG_\mathrm{\!hol}(M, L))$, the space of intertwining
operators from the irreducible representation $V_\Ga$ with highest weight
$\Ga$ to the quantization $\GG_\mathrm{\!hol}(M, L)$ of $M$.
\end{Thm}

\begin{Rem}
\renewcommand{\labelenumi}{(\theenumi)}
\begin{enumerate}
\item Of course, {\em dominance} of a weight (an analytically integral
	element of $\Ft^*$) is defined w.\,r.\,t.\ the choice of a Weyl
	chamber. It is introduced in the statement of the preceding theorem
	only to avoid redundancy since one has $M_{\Ad^*(g)\Ga} \~ M_\Ga$
	and $V_{\Ad^*(g)\Ga} \~ V_\Ga$. Thus one may write here {\em
	integral point of $\Fg^*$} in the sense of
	Def.\,\ref{DEF:integrelem} as well.
\item The reduced spaces $M_\Ga$ inherit, also in the singular case,
	``sufficient'' structure in order to define holomorphic sections of
	$L_\Ga \RA M_\Ga$ (compare \cite{Sj-redmult}). By
	Theorem~\ref{TH:orbcorr} and Remark~\ref{RK:freeact}(2) the reduced
	spaces occurring in our ``symplectic Howe setting'' are always
	smooth. Thus, in fact, we do not need the results of Sjamaar in its
	full generality.
\item If $\Ga$, interpreted as an element of $\Fg^*$, is not in the moment
	image $\Phi(M)$, then $\GG_\mathrm{\!hol}(M_{\Ga}, L_{\Ga})$ is
	to be interpreted as $\{0\}$. Theorem~\ref{TH:QRcomm} then says that
	$V_\Ga$ does not occur in the $G$-decomposition of
	$\GG_\mathrm{\!hol}(M, L)$.
\end{enumerate}
\end{Rem}

Take now the simultaneous $G_1 \times G_2$-action introduced in the previous
section. By the Borel-Weil theorem, the geometric quantizations, $V_{\Ga_i} \~
\GG_\mathrm{\!hol}(\orb_{\Ga_i}, L_{\Ga_i})$ with $\Ga_i$ integral, realize the
irreducible representations of $G_i$, thus  $\GG_\mathrm{\!hol}(\orb_{\Ga_1}
\times \orb_{\Ga_2}, L_{\Ga_1} \boxtimes L_{\Ga_2})$ the irreducibles for
$G_1 \times G_2$. Here $L_{\Ga_1} \boxtimes L_{\Ga_2} \RA M_1 \times M_2$ is
given as $p_1^*(L_{\Ga_1}) \otimes p_2^*(L_{\Ga_2})$ with $p_j : M_1 \times
M_2 \RA M_j$ denoting the $j$-th projection for $j = 1,2$. Now,
\begin{multline*}
	\HOM_{G_1\times G_2}\PAR{
		\GG_\mathrm{\!hol}(\orb_{\Ga_1} \times \orb_{\Ga_2},
					L_{\Ga_1} \boxtimes L_{\Ga_2}),
		\GG_\mathrm{\!hol}(M, L)}\\
	\~ \GG_\mathrm{\!hol}\PAR{
		\Phi^{-1}(\orb_{\Ga_1} \times \orb_{\Ga_2})/(G_1 \times G_2),
		 L_{(\Ga_1,\Ga_2)}}.
\end{multline*}
By Thm.\,\ref{TH:orbcorr}(iv), the reduced space may be empty (if the
coadjoint orbits are not in correspondence), hence the space of sections
$\GG_\mathrm{\!hol}(\Phi^{-1}(\orb_{\Ga_1} \times \orb_{\Ga_2})/
(G_1 \times G_2), L_{(\Ga_1,\Ga_2)})$ is trivial; otherwise the reduced space
is a point and the space of sections is simply $\C$. So, one concludes for
the multiplicities of $G_1\times G_2$-representations in the quantization of
$M$:
\[
	\dim \HOM_{G_1\times G_2}\PAR{
		\GG_\mathrm{\!hol}(\orb_{\Ga_1} \times \orb_{\Ga_2},
					L_{\Ga_1} \boxtimes L_{\Ga_2}),
		\GG_\mathrm{\!hol}(M, L)} \le 1,
\]
the equal sign being true if and only if $\orb_{\Ga_2} = \GL(\orb_{\Ga_1})$.

~\\

Interpreting the line bundles $L_{\Ga_1},  L_{\Ga_2}$ and $L_{\Ga_1}
\boxtimes L_{\Ga_2}$ as sheaves and their holomorphic sections as their
zeroth cohomology group, one concludes from a standard K\"unneth formula
(compare \cite{SampWash} and \cite{Kaup}) that
\[
\GG_\mathrm{\!hol}(\orb_{\Ga_1} \times \orb_{\Ga_2},
					L_{\Ga_1} \boxtimes L_{\Ga_2})
	\~ \GG_\mathrm{\!hol}(\orb_{\Ga_1}, L_{\Ga_1})
	\otimes \GG_\mathrm{\!hol}(\orb_{\Ga_2}, L_{\Ga_2}).
\]

Recall that by \ref{TH:orbcorr}(iii), the coadjoint $G_2$-orbit corresponding
via $\GL$ to $\orb_{\Ga_1} \in \Phi_1(M)/G_1$ is symplectomorphic to the orbit
reduced space of $\orb_{\Ga_1}$, \ie
\[
	\GL(\orb_{\Ga_1}) \~ \Phi_1^{-1}(\orb_{\Ga_1})/G_1 = M_{\Ga_1}.
\]

This implies that the multiplicity space of one action is an
irreducible representation of the other action, \ie for $\orb_{\Ga_2} =
\GL(\orb_{\Ga_1})$, one has $G_2$-equivariantly
\[
	\HOM_{G_1}(V_{\Ga_1}, \GG_\mathrm{\!hol}(M, L)) \~
		\GG_\mathrm{\!hol}(M_{\Ga_1}, L(M_{\Ga_1})) \~
			\GG_\mathrm{\!hol}(\orb_{\Ga_2}, L_{\Ga_2}).
\]

Let us identify in the sequel (for $k = 1,2$) $\widehat{G_k}$, the set of
equivalence classes of irreducible complex representations of finite
dimension of $G_k$, with $(\Ft_k)_\Z^+$ the integral points in a fixed Weyl
chamber $\Ft_k^+$ in the dual $\Ft_k^*$ of a maximal abelian subalgebra
$\Ft_k \SSE \Fg_k$.

The preceding statements can be summarized as follows.
\begin{Thm}
\label{TH:qucorr}
Let $G_1$ and $G_2$ be compact connected Lie groups acting by holomorphic
transformations on a K\"ahler manifold $M$ such that the actions extend to
actions of the respective complexified groups. Suppose that the actions of
$G_1$ and $G_2$ commute and are hamiltonian with admissible equivariant
moment maps $\Phi_1$ and $\Phi_2$. Denote by $L$ the prequantum line bundle
over $M$.

Assume the symplectic Howe condition to be satisfied. Then:
\[
	\GG_\mathrm{\!hol}(M, L) \~
		\complsum_{\Ga_1 \in \Phi_1(M) \cap \widehat{G_1}}
			\GG_\mathrm{\!hol}(\orb_{\Ga_1}, L_{\Ga_1})
			\otimes \GG_\mathrm{\!hol}(\orb_{\Ga_2}, L_{\Ga_2}),
\]
where $\GL$ is the orbit correspondence map and $\orb_{\Ga_2} =
\GL(\orb_{\Ga_1})$.
\end{Thm}

\begin{Rem}
\renewcommand{\labelenumi}{(\theenumi)}
\begin{enumerate}
\item In the preceding theorem the symbol $\widehat{\bigoplus}$ of course
	signifies the completion of the algebraic direct sum in the
	Fr\'echet topology of the space of holomorphic sections of $L$.
\item Since $\GL$, viewed as a map from $\Phi_1(M) \cap \widehat{G_1}$ to
	$\widehat{G_2}$, is injective the representation 
	$\GG_\mathrm{\!hol}(M, L)$ of the pair $(G_1, G_2)$ satisfies the
	representation-theoretic Howe condition $(*)$ recalled in the
	introduction.
\item Let $(M, \Go)$ be a compact complex manifold together with an integral
K\"ahler form and $L \RA M$ a prequantizing holomorphic line bundle. Let
furthermore the connected compact Lie groups $G_1$ and $G_2$ act by
holomorphic transformations and in a hamiltonian fashion on $M$. If this
pair of actions satisfies the symplectic Howe condition, the preceding theorem
applies, \ie their induced linear representation on the ``holomorphic
quantization'' $\GG_\mathrm{\!hol}(M, L)$ is equipped with a Howe duality.
\end{enumerate}
\end{Rem}

\section{Examples}
In this section we give some simple but instructive examples of pairs 
of symplectic actions fulfilling the symplectic Howe condition. In each case, we explicitly give
the orbit correspondence map and the Howe duality map
associated to the linear representation of the pair of groups on the
geometric quantization space.

\subsection{$(U(n), U(m))$ on $\Mat(n {\times} m; \C)$}
Let $M = \Mat(n  {\times} m; \C)$ be the $n{\times} m$ matrices with complex entries
equipped with the symplectic structure $\Go(A, B) = \Im \tr (\bar{A}^\Tr B)$
and the natural actions of $G_1=U(n)$ and $G_2=U(m)$: $U(n) \times U(m) \times M \RA
M, ((U, V), z) \MT UzV^{-1}$. Assume that $n \ge m \ge 1$. The (equivariant) 
moment maps corresponding to these actions are given by $\Phi_1^\xi(z) =
-\OH \Im \tr (\xi z \bar{z}^\Tr)$ (for $\xi \in \Fu(n)$) and $\Phi_2^\eta(z)
= \OH \Im \tr (\eta \bar{z}^\Tr z)$ (for $\eta \in \Fu(m)$).
Direct inspection shows that any level set of $\Phi_1$ is exactly one
$U(m)$-orbit. Using standard facts from invariant theory and a theorem of
G.\,Schwarz (see \cite{Schw-invts}, Thm.\,1) one obtaines that the space of pullbacks
of smooth functions on $\Fg_1^*$ under $\Phi_1$ coincides with the algebra
of $G_2$-invariant smooth functions on $M$, and thus by
Lemma~\ref{LEM:centr-invt} with the centralizer of the $\Phi_2$-pullback
of $\Cinf(\Fg_2^*)$ in the Poisson algebra $\Cinf(M)$. The analogous
statements hold if the roles of the two groups are reversed. Therefore,
the symplectic Howe condition is satisfied.\\

In order to state the orbit correspondence of Thm.\,\ref{TH:orbcorr}
explicitly, recall that any $z \in M$ can be written
$z = U\GS V$ for
\[
	\GS = \GS(\Gs_1, \ldots, \Gs_m) =
	\begin{pmatrix}\Gs_1& 0& \cdots& 0\\
			0& \Gs_2& \cdots& 0\\
			\vdots& \vdots& \ddots& \vdots\\
			0& 0& \cdots& \Gs_m\\
			0& 0& \cdots& 0\\
			\vdots& \vdots& \vdots& \vdots
	\end{pmatrix} \in M,
\]
where $\Gs_1 \ge \Gs_2 \ge \ldots \ge \Gs_m \ge 0$ and $U \in U(n), V \in
U(m)$ are determined up to an element of the stabilizer of $\GS$. The set $S = 
\{\Gs = (\Gs_1, \ldots, \Gs_m) \:|\: \Gs_1 > \Gs_2 > \ldots > \Gs_m > 0\}$
parametrizes thus all $U(n)\times U(m)$-orbits of maximal dimension and its closure
$\bar{S}$ all $U(n)\times U(m)$-orbits.  Identifying the adjoint and the coadjoint action
of a unitary group via the usual trace form isomorphism ($\zeta\mapsto 
\tr(\cdot \, \zeta)$) we get $\widetilde{\Phi}_1(z) =\IO{2} z \bar{z}^\Tr$
and $\widetilde{\Phi}_2(z) =-\IO{2} \bar{z}^\Tr z$ fulfilling
$\Phi_1^\xi(z) =\tr (\xi \cdot \widetilde{\Phi}_1(z) )$ and 
$\Phi_2^\eta(z)=\tr (\eta \cdot \widetilde{\Phi}_2(z))$. 
Denoting furthermore the diagonal element of $\Fu(N)$ with entries 
$\I (\Gl_1, \ldots, \Gl_N)$ by $[\Gl_1, \ldots, \Gl_N]$
or simply by $[\Gl]$
we have here the following orbit correspondence map:
\[
\GL : \orb^{U(n)}_{[{1\over 2}\Gs_1^2, \ldots, {1\over 2}\Gs_m^2, 0, \ldots, 0]} \MT
\orb^{U(m)}_{[-{1\over 2}\Gs_1^2, \ldots, -{1\over 2}\Gs_m^2]} \,\, .
\]
Observe that orbits of any dimension are included in this statement
and no restriction to the principal orbit type is necessary.
For the sake of brevity we write $[\Ga,0]$ for 
$[\Ga_1, \ldots, \Ga_m,0, \ldots ,0]$ in the sequel.\\

Now, we may apply Thm.\,\ref{TH:qucorr} and obtain, since the prequantum
line bundle $L$ over $M$ is holomorphically trivial,
\[
\mathrm{Hol}(M)=	\GG_\mathrm{hol}(M, L) = \complsum_{\Ga \in \bar{S}\,\mathrm{integral}}
		\GG_\mathrm{hol}(\orb^{U(n)}_{[\Ga, 0]}, L_{[\Ga, 0]})
		\otimes \GG_\mathrm{hol}(\orb^{U(m)}_{[-\Ga]}, L_{[-\Ga]})
\]
\[		
		\cong \complsum_{\Ga \in \bar{S}\,\mathrm{integral}}
		V_{[\Ga, 0]} \otimes \left(V_{[\Ga]}\right)^*,
\]
where for $\Gl=(\Gl_1, \ldots, \Gl_N)$ (with $\Gl_1,\ldots \Gl_n$ integral and
$\Gl_1 \ge \Gl_2 \ge \ldots \ge \Gl_N$), $V_{[\Gl]}$ is the irreducible 
$U(N)$-representation associated to the orbit through
$\I (\Gl_1, \ldots, \Gl_N)$ either by the orbit method or equivalently by the 
highest-weight procedure. 
The ensuing Howe duality map,  $V_{[\Ga, 0]} \mapsto \left(V_{[\Ga]}\right)^*$,
is then essentially a restatement of the well-known $GL(n)$-$GL(m)$-duality
(compare, \eg Thm.\,5.2.7 in \cite{GoodWall}) on the space of holomorphic
polynomials on $M$.

\subsection{$(G, G)$ on $T^*G$}
Given a Lie group $G$, the left and right action of $G$ on itself of course
lift to its cotangent bundle $T^*G$. Both actions are free and proper, but
this is not in general true for the product. Trivializing $T^*G$ via the
right action, the left resp. right action reads as follows: $(g, \Ga) \MT
(hg, \Ad^*(h)\Ga)$ and $(g, \Ga) \MT (gh^{-1}, \Ga)$, respectively.
Furthermore, the moment maps are then given by $\Phi_1 : (g, \Ga) \MT \Ga$
and $\Phi_2 : (g, \Ga) \MT -\Ad^*(g^{-1})\Ga$ for the left and right action,
respectively. Obviously, the $\Phi_1$-fibres are right-$G$-orbits and the
$\Phi_2$-fibres left-$G$-orbits. For connected $G$ the symplectic Howe
condition can easily be verified using Lemma~\ref{LEM:centr-invt};
the orbit correspondence is $\GL : \orb_\Ga \MT \orb_{-\Ga}$.

Restricting now to the case that $G$ is compact and connected, $T^*G$ can be
identified with the affine K\"ahler manifold $G^\C$ (see, \eg Sect.\,3 in 
\cite{Hall-CMPh184}) and the holomorphic
geometric quantization $\Hol(G^\C)$ has the same $(G\times G)$-finite
vectors as the representation on $L^2(G)$. In this situation,
Thm.\,\ref{TH:qucorr} specializes thus to the Peter-Weyl theorem.

\subsection{$(U(1), U(n))$ on $\PS_n(\C)$}
The standard $U(n{+}1)$-action on $\PS_n(\C)$ preserves the complex structure
and the scaled Fubini-Study K\"ahler form $k\Go^\mathrm{FS}$ ($k \in \N
\backslash \{0\}$), and possesses the
following equivariant moment map (for $\Gz \in \Fu(n{+}1)$):
\[
	\Phi^\Gz([z]) = k\IO{2\pi} \frac{\ANG{z, \Gz \cdot z}}{\ANG{z, z}}.
\]
The injections $U(n) \RA U(n{+}1), A \MT \begin{pmatrix} 1 & 0 \\ 0 & A
\end{pmatrix}$ and $U(1) \RA U(n{+}1), e^{2\pi\I\Gt} \MT \begin{pmatrix}
e^{2\pi\I\Gt} & 0 \\ 0 & \mathbf{1}_n \end{pmatrix}$ yield a pair of
commuting symplectic actions with induced moment maps. A study
of invariant functions on $\PS_n(\C)$ shows that the action of the pair
$(U(1), U(n))$ satisfies the symplectic Howe condition. 

Fixing $\xi_0 = 2\pi\I$ the moment map $\Phi_1$ reads as $\Phi^{\xi_0} \cdot
\xi_0^* \in \Fu(1)^*$ and -identifying $\Fu(1)^*$ with $\R$ via the base
consisting of $\xi_0^*$- one has that $\Phi_1(M) = [-k, 0]$. We also find
that $\Phi_2(M) = U(n)\cdot S$, where $S = \{y \cdot \Gf_{11} \:|\: 0 \le y
\le k\}$ is a global slice, with $\Gf_{11}(X) = \IO{2\pi} X_{11}$ for $X \in
\Fu(n) \SSE \Mat(n {\times} n; \C)$. The orbit correspondence $\GL$ is then given by
$x \MT U(n)\cdot ((x+k)\Gf_{11})$. Integral points in $\Phi_1(M)$ are simply
points $m \in [-k, 0] \cap \Z$ with associated representation
$e^{2\pi\I\Gt} \MT e^{2\pi\I m\Gt}$, whereas in $\Phi_2(M)$ an orbit
$U(n)\cdot (y\Gf_{11})$ with $y \ge 0$ is integral if and only if
$y \in \N$. Obviously, integrality of orbits in the sense of
Def.\,\ref{DEF:integrelem} is preserved by the orbit correspondence map
$\GL$.

For $k > 0$, the
geometric quantization of $(\PS_n(\C), k\Go^\mathrm{FS})$ is given
by the holomorphic sections module of the $k$-th power of the hyperplane
bundle over $\PS_n(\C)$, isomorphic as a $U(n{+}1)$-representation to
$\C_k[z_0, z_1, \ldots, z_n]$, the space of complex homogeneous polynomials
of degree $k$ in $n{+}1$ variables.

As a $(U(1) \times U(n))$-representation this space decomposes as
\[
	\bigoplus_{d=0}^k \C_d[z_0] \otimes \C_{k-d}[z_1, \ldots, z_n],
\]
where for $l \in \N$, $\C_l[z_1, \ldots, z_n]$ is the space of complex
homogeneous polynomials of degree $l$ in $n$ variables, the
$U(n)$-representation associated to $U(n)\cdot (l \Gf_{11}) \in 
\Phi_2(M)/U(n)$, and $\C_d[z_0]$ realizes the $U(1)$-representation on $\C$ given
by the character $e^{2\pi\I\Gt} \MT e^{2\pi\I(-d)\Gt}$ associated to
$-d \in \Phi_1(M) \cap \Z$. The representation-theoretic Howe duality map
of this $(U(1) \times U(n))$-representation, $\C_d[z_0] \MT
\C_{k-d}[z_1, \ldots, z_n]$, is then clearly induced by the orbit
correspondence map.

\section*{Acknowledgements}
We would like to thank Joachim Hilgert for several helpful discussions related to this work. The second
named author would like to thank as well Marcus J. Slupinski for anterior discussions on the notion
of Howe duality.

\addcontentsline{toc}{section}{References}

\begin{thebibliography}{KKS78}

\bibitem[Ada87]{Adams-orb-dual}
Jeffrey Adams.
\newblock Coadjoint orbits and reductive dual pairs.
\newblock {\em Adv.\,in\,Math.}, 63:138--151, 1987.

\bibitem[Dui96]{Duis-Spinc}
Johannes~J. Duistermaat.
\newblock {\em The heat kernel {L}efschetz fixed point formula for the
  spin{$^c$} {D}irac operator}.
\newblock Progress in Nonlinear Differential Equations and their Applications,
  18. Birkh\"auser Boston Inc., Boston, MA, 1996.

\bibitem[GS82]{GS-geomqu}
Victor Guillemin and Shlomo Sternberg.
\newblock Geometric quantization and multiplicities of group representations.
\newblock {\em Invent.\,Math.}, 67:515--538, 1982.

\bibitem[GS84]{GS-mfsp}
Victor Guillemin and Shlomo Sternberg.
\newblock Multiplicity-free spaces.
\newblock {\em J.\,Differential Geometry}, 19:31--56, 1984.

\bibitem[GW98]{GoodWall}
Roe Goodman and Nolan~R. Wallach.
\newblock {\em Representations and invariants of the classical groups},
  volume~68 of {\em Encyclopedia of Mathematics and its Applications}.
\newblock Cambridge University Press, Cambridge, 1998.

\bibitem[Hal97]{Hall-CMPh184}
Brian~C. Hall.
\newblock Phase space bounds for quantum mechanics on a compact {L}ie group.
\newblock {\em Comm. Math. Phys.}, 184(1):233--250, 1997.

\bibitem[How89]{Howe-remarks}
Roger Howe.
\newblock Remarks on classical invariant theory.
\newblock {\em Trans. Amer. Math. Soc.}, 313(2):539--570, 1989.

\bibitem[HW90]{HW-coisoact}
Alan~T. Huckleberry and Tilmann Wurzbacher.
\newblock Multiplicity-free complex manifolds.
\newblock {\em Math. Ann.}, 286(1-3):261--280, 1990.

\bibitem[Kau67]{Kaup}
Ludger Kaup.
\newblock Eine {K}\"unnethformel f\"ur {F}r\'echetgarben.
\newblock {\em Math. Z.}, 97:158--168, 1967.

\bibitem[KKS78]{KKS-Calo}
D.~Kazhdan, B.~Kostant, and S.~Sternberg.
\newblock Hamiltonian group actions and dynamical systems of {C}alogero type.
\newblock {\em Comm. Pure Appl. Math.}, 31(4):481--507, 1978.

\bibitem[KL97]{KL-centr}
Yael Karshon and Eugene Lerman.
\newblock The centralizer of invariant functions and division properties of the
  moment map.
\newblock {\em Illinois J. Math.}, 41(3):462--487, 1997.

\bibitem[Kos70]{Kost-preq}
Bertram Kostant.
\newblock Quantization and unitary representations. {I}. {P}requantization.
\newblock In {\em Lectures in modern analysis and applications, III}, pages
  87--208. Lecture Notes in Math., Vol. 170. Springer, Berlin, 1970.

\bibitem[OR04]{OR-moment}
Juan-Pablo Ortega and Tudor~S. Ratiu.
\newblock {\em Momentum maps and {H}amiltonian reduction}, volume 222 of {\em
  Progress in Mathematics}.
\newblock Birkh\"auser Boston Inc., Boston, MA, 2004.

\bibitem[RN76]{Bohr-coll3}
L\'eon Rosenfeld and Jens~Rud Nielsen\:(eds.).
\newblock {\em The {C}orrespondence {P}rinciple (1918--1923)}, volume~3 of {\em
  Niels {B}ohr {C}ollected {W}orks}.
\newblock North-Holland, Amsterdam, 1976.

\bibitem[Sch75]{Schw-invts}
Gerald~W. Schwarz.
\newblock Smooth functions invariant under the action of a compact {L}ie group.
\newblock {\em Topology}, 14:63--68, 1975.

\bibitem[Sja95]{Sj-redmult}
Reyer Sjamaar.
\newblock Holomorphic slices, symplectic reduction and multiplicities of
  representations.
\newblock {\em Ann. of Math. (2)}, 141(1):87--129, 1995.

\bibitem[SW59]{SampWash}
Joseph~H. Sampson and Gerard Washnitzer.
\newblock A {K}\"unneth formula for coherent algebraic sheaves.
\newblock {\em Illinois J. Math.}, 3:389--402, 1959.

\bibitem[Wei83]{Wein-locstr}
Alan Weinstein.
\newblock The local structure of {P}oisson manifolds.
\newblock {\em J. Differential Geom.}, 18(3):523--557, 1983.

\end{thebibliography}

\end{document}